\documentclass{amsart}

\usepackage{epstopdf}
\usepackage{subfigure}

\theoremstyle{plain}
\newtheorem{theorem}{Theorem}[section]

\newtheorem{lemma}[theorem]{Lemma}

\newtheorem{prop}[theorem]{Proposition}

\theoremstyle{definition}

\newtheorem{rema}[theorem]{Remark}
\newtheorem{defi}[theorem]{Definition}

\theoremstyle{remark}

\newcommand{\al}{\alpha}
\newcommand{\be}{\beta}
\newcommand{\ga}{\gamma}

\newcommand{\ep}{\epsilon}

\newcommand{\pbp}{\partial \bar{\partial}}
\newcommand{\ddc}{d d^{c}}

\newcommand{\tr}{\operatorname{tr}}

\begin{document}

\title{A priori estimates for a generalised Monge-Amp\`ere PDE on some compact K\"ahler manifolds}

\author[Pingali]{Vamsi P. Pingali}
\address{Department of Mathematics\\
412 Krieger Hall, Johns Hopkins University,\\ Baltimore, MD 21218, USA}
\email{vpingali@math.jhu.edu}

\maketitle

\begin{abstract}
We study a fully nonlinear PDE involving a linear combination of symmetric polynomials of the K\"ahler form on a K\"ahler manifold. A $C^0$ \emph{a priori} estimate is proven in general and a gradient estimate is proven in certain cases. Independently, we also provide a method-of-continuity proof via a path of K\"ahler metrics to recover the existence of solutions in some of the known cases. Known results are then applied to an analytic problem arising from Chern-Weil theory and to a special Lagrangian-type equation arising from mirror symmetry.
\end{abstract}

\section{Introduction}
Consider the following fairly general equation on a compact K\"ahler manifold $(X,\omega)$.
\begin{gather}
\displaystyle \omega_{\phi} ^n = \sum_{k=1}^{n} \alpha_k \wedge \omega _{\phi}^{n-k}, \label{maineq}
\end{gather}
where $\alpha _k \geq 0$ are closed smooth positive $(k,k)$-forms such that $\alpha _{k_0} >0$ for at least one $k_0$ (in particular, $k_0$ can be equal to $n$), $\phi$ is a smooth function such that $\omega _{\phi} = \omega + \sqrt{-1} \pbp \phi >0$, and $\omega$ satisfies the so-called ``cone condition" $\displaystyle n\omega ^{n-1} - \sum_{k=1}^{n} (n-k) \alpha_k \wedge \omega ^{n-k-1} > 0$ and the consistency condition $\displaystyle \int \omega ^n = \sum _{k=1}^n \int \alpha _k \wedge \omega ^{n-k}$. Our notion of positivity of $(p,p)$-forms is explained in section \ref{setup}. \\
\indent Notice that if $\alpha _k = \delta _{kn} \eta$ then equation \ref{maineq} boils down to the Calabi-Yau theorem \cite{Yau}. In its full generality, equation \ref{maineq} and its cousins arise in the representation problem of Chern-Weil theory \cite{MAtorus}, canonical metrics in K\"ahler geometry \cite{Sze,Weisun,InvHess,Hess1,Hess2,XX,Coll}, symplectic geometry \cite{yanirub} and mirror symmetry in string theory \cite{special,leung}. So far, in most places where it has been studied a flow technique (like the J-flow) was used to study it. In Wei Sun's paper \cite{Weisun} the method of continuity was used but it used non-K\"ahler metrics to prove openness. In this paper we aim to prove \emph{a priori} estimates and solve \ref{maineq} in some cases. \\
\indent Our first result is an \emph{a priori} $C^0$ estimate on $\phi$ under general assumptions (proposition \ref{uniformestimate}). We remark that a paper of Sz\'ekelyhidi \cite{Sze} deals with an equation that overlaps with the generalised Monge-Amp\`ere equation in some cases (like the J-flow case). In \cite{Sze} the ABP estimate is used to prove a $C^0$-estimate but we use Yau's Moser iteration argument. In the course of setting up the method of continuity we indicate a proof of a theorem (theorem $1.4$ in \cite{Weisun} when $\psi$ is a constant which is also a theorem of Fang-Lai-Ma \cite{InvHess})  in remark \ref{Weithmpf} using a method of continuity that uses only K\"ahler metrics. This may potentially be of independent interest. Then we proceed to prove a gradient estimate for \ref{maineq} in a special case.
\begin{theorem}\label{mainthm}
Let $(X,\omega)$ be a compact K\"ahler manifold. Assume that $\chi$ is another K\"ahler form such that $\chi$ has non-negative bisectional curvature. Assume that $\alpha = f \chi^{n-1} + \left(\sqrt{-1}\right)^{n-1} (-1)^{n(n-1)/2}\displaystyle \sum _{a=1} ^N f_a \Phi _a \wedge \bar{\Phi}_a$ is a closed smooth $(n-1,n-1)$-form ($N$ is an arbitrary natural number) which satisfies $-C\alpha  \leq  \nabla _X \alpha \leq C \alpha$ for all real $\chi$-unit vectors $X$, where $f \geq 0, f_a \geq 0$ are smooth functions, $\Phi _a$ are smooth $(n-1,0)$-forms, $C>0$ is a constant, and $\nabla$ is the canonical connection induced by $\chi$. Let $\eta >0$ be an $(n,n)$-form. Assume that $\displaystyle \int \omega ^n = \int \alpha \wedge \omega + \int \eta$ and that $n\omega ^{n-1} - \alpha > 0$. For a smooth function $\phi$, denote $\omega _{\phi} = \omega + \sqrt{-1} \pbp \phi$. Consider the equation.
\begin{gather}
\omega _{\phi} ^{n} = \alpha \wedge \omega _{\phi} + \eta. \label{maineqthm}
\end{gather}
The following hold.
\begin{enumerate}
\item \emph{Gradient estimate} A smooth solution $\phi$ of equation \ref{maineqthm} satisfies $\Vert \phi \Vert _{C^1}\leq C$ where $C$ depends only on the coefficients.
\item \emph{Partial Laplacian estimate and existence}  If $\alpha$ is parallel with respect to $\chi$ then $\frac{\alpha \wedge \omega _{\phi}}{\chi ^n} \leq C$. In addition, if $\alpha >0$, i.e., if $f>0$, then a unique smooth solution exists satisfying $\omega _{\phi} >0$ and $n\omega _{\phi}^{n-1}  -\alpha >0$.
\end{enumerate}
\end{theorem}
\begin{rema}\label{mainthemrem}
\indent The existence part of theorem \ref{mainthm} actually follows from a far more general theorem of Wei Sun \cite{Weisun}. However, the \emph{a priori} estimates are new in the case of $\alpha$ being degenerate. This is perhaps the main point of the theorem. The assumptions of theorem \ref{mainthm} might look restrictive but actually even as stated the technique used to prove them might potentially be useful in \cite{yanirub} where similar assumptions are in force (essentially \cite{yanirub} deals with domains in $\mathbb{R}^n$ but the difficulty is that the equation there is degenerate elliptic).
\end{rema}
\indent  In situations involving equations like \ref{maineq} dealing with the J-flow or in the special case of \ref{maineq} mentioned in remark \ref{Weithmpf} one can actually avoid the gradient estimate by proving the laplacian estimate directly. In fact, thanks to the work of Tosatti and Weinkove \cite{TWb}, just proving the estimate $\Delta \phi \leq C e^{A(\phi - \inf \phi)}$ is enough to guarantee a $C^2$ bound on $\phi$. Usually the technique behind proving such estimates is to use the maximum principle on an appropriately chosen function. For instance, one choice \cite{Weisun} is $\psi = e^w (\Delta_{\chi} \phi + \tr_{\chi} \omega)$ where $w = - A \phi + f(\phi)$ is chosen judiciously. This method was pioneered by Aubin \cite{Aub} and Yau \cite{Yau}. The major difficulty here is that in general, equation \ref{maineq} is not a symmetric polynomial in the hessian. This problem is exacerbated if we allow $\alpha_k$ to be degenerate. Therefore it is not clear that some inequalities in the spirit of \cite{Weisun,Coll,InvHess} work in this setting.\\
\indent Independently, we apply the main result in \cite{Weisun} to prove two theorems. The first one deals with Chern-Weil theory.
\begin{theorem}\label{Coro1}
Let $(V,h_0)$ be a hermitian rank-$k$ holomorphic vector bundle over a compact K\"ahler manifold $(X,\omega)$. Denote the curvature of the Chern connection of $h_0$ by $F_0$ and define $\Theta _0 = \frac{\sqrt{-1}F_0}{2\pi}$. Let $\eta$ be an $(n,n)$-form on $X$ representing the top Chern character class, i.e., $[\eta] =\left[ \mathrm{tr}\left (\Theta _0 \right)^n \right ]$. Define the forms $\alpha _i$ inductively according to
\begin{gather}
\alpha _ 1 = n \left(\omega - \frac{1}{k}\mathrm{tr}(\Theta _0)\right) \nonumber \\
\alpha _ p = - \dbinom{n}{p}  \frac{1}{k}\tr(\Theta _0)^ p + \dbinom{n}{p} \omega ^p - \displaystyle \sum _{i=1} ^{p-1} \dbinom{n-i}{p-i} \alpha _ i  \omega ^{p-i} \ \forall \ 2\leq p \leq n-1 \nonumber \\
\alpha _n = \frac{\eta}{k} -  \tr(\Theta _0)^ n +  \omega ^n - \displaystyle \sum _{i=1} ^{n-1}  \alpha _ i  \omega ^{n-i} .\label{Cherneq}
\end{gather}
Assume that there exists a hermitian metric $\chi$, constants $c_i \geq 0$, and a smooth function $\psi$ such that $\alpha _i = c_i \chi ^i \psi$ and $\sum c_i >0$. Also assume that $\omega$ satisfies $n\omega ^{n-1} - \displaystyle \sum _{i=}^{n-1}(n-i)\alpha _i \omega ^{n-i-1} >0$. Then there exists a smooth metric $h=h_0 e^{-2\pi \phi}$, unique upto constant multiples such that the top Chern character form of $h$ is $\eta$, i.e., $\eta = \mathrm{tr}\left (\Theta _{h} \right)^n $.
\end{theorem}
\noindent Some examples of the applicability of a very restricted version of theorem \ref{Coro1} are given in \cite{MAtorus}. \\
\indent The second one deals with a special-Lagrangian type equation motivated from mirror symmetry.
\begin{theorem}\label{Coro2}
Let $L$ be a holomorphic line bundle over a compact K\"ahler manifold $(X,\omega)$. Let $\hat{\theta}$ be defined by the equation $Im([\omega] +2\pi \sqrt{-1} c_1 (L))^3 = \tan (\hat{\theta}) Re ([\omega] + 2\pi \sqrt{-1} c_1 (L))^3$. Assume that $\tan (\hat{\theta})>0$. Also assume that there exists a metric $h_0$ on $L$ whose curvature $F_0$ is such that the $(1,1)$-form $\Omega = \sqrt{-1}F_0-\omega \tan (\hat{\theta})$ satisfies
\begin{enumerate}
\item $\Omega >0$, and
\item $\Omega ^2 - \omega ^2 \sec ^2{\hat{\theta}} >0$.
\end{enumerate}
Then there exists a smooth metric $h=h_0 e^{-\phi}$, unique upto constant multiples satisfying
\begin{gather}
Im \left(\omega - F_{\phi}\right)^3 = \tan (\hat{\theta}) Re \left(\omega - F_{\phi}^3 \right ). \label{jacobeq}
\end{gather}
\end{theorem}
We remark that since the theorem \ref{Coro2} does not require non-negative bisectional curvature, it is in some cases more general than the result in \cite{special}. In particular, it may be applied to the Calabi-Yau $3$-folds that are of interest to physicists. We give an example in section \ref{app}.\\
\indent Here is a more detailed outline of the paper. In section \ref{setup} we set up the method of continuity, prove uniqueness and a uniform estimate, indicate a proof of the theorem in \cite{Weisun}, and also prove that upper bounds on $\omega _{\phi}$ lead to uniform ellipticity. Owing to the non-symmetric nature of the equation, this is actually somewhat nontrivial. In sections \ref{gradi} and \ref{high} we prove further \emph{a priori} estimates in the special case of the equation in theorem \ref{mainthm}. In section \ref{app} we prove theorems \ref{Coro1} and \ref{Coro2}. \\

\section{Setup of the method of continuity and the uniform estimate}\label{setup}
\indent Before proceeding further, we define a notion of positivity of $(p,p)$-forms.
\begin{defi}
Let $(X,\chi)$ be a hermitian manifold. A smooth $(p,p)$-form $\alpha_p$ is (strictly) positive if $\alpha _p = f \chi ^p + (\sqrt{-1})^p (-1)^{p(p-1)/2}\displaystyle \sum _{k=1} ^N f_k \Phi _{k} \wedge \overline{\Phi}_{k}$ where $f$ is a (strictly) positive smooth function, $f_k \geq 0$ are positive \footnote{Unless specified otherwise, we use positive in the french sense to mean non-negative} smooth functions, and $\Phi_{k}$ are smooth $(p,0)$-forms. Moreover, we write $\alpha \geq 0$ if $\alpha$ is positive and $\alpha > 0$ if it is strictly so.
\label{positivity}
\end{defi}
\begin{rema}\label{generalpositive}
Perhaps a more natural definition would be to require that $\alpha_p$ define a hermitian non-negative bilinear form on $\Lambda ^p T^{(1,0)} X$. However, one can easily see that this is equivalent to our definition. In particular, the wedge product of strictly positive forms is strictly positive.
\end{rema}
\indent In order to solve \ref{maineq} we employ the method of continuity. In whatever follows we assume that on $(X,\omega)$ for at least one value of $k_0$, $\alpha _{k_0} > \delta \omega ^{k_0}$ for a positive constant $\delta$.  Consider the following family of equations parametrised by $t\in[0,1]$.
\begin{gather}
\omega_{\phi _t} ^n = t\sum_{k=1}^{n-1} \alpha_k \wedge \omega _{\phi _t}^{n-k} + \alpha_n b_t c^{1-t},
\label{conti}
\end{gather}
where $\omega_{\phi _t} = \omega + \sqrt{-1} \pbp \phi _t$ is a K\"ahler form, $\alpha _k$ are $d$-closed positive $(k,k)$-forms and $\alpha _n$ is a strictly positive $(n,n)$-form, $c =\displaystyle \frac{\int \omega ^n}{\int \alpha _n}$, and $b_t$ is a normalising constant chosen so that the integrals are equal on both sides, i.e., $b_t =  c^{t-1}\frac{\displaystyle \int \left(\omega ^n - t  \displaystyle\sum _{k=1} ^{n-1}\alpha _k \wedge \omega ^{n-k}\right)}{\displaystyle\int \alpha _n}= c^{t-1}\frac{\displaystyle \int \left((1-t)\omega ^n + t  \displaystyle \alpha _n\right)}{\displaystyle\int \alpha _n}$. Thus $b_t c^{1-t} \geq 1$. \\
\indent Let $\mathcal{T}$ be the set of $t\in [0,1]$ where equation \ref{conti} has a unique smooth solution $\phi _t$ such that $\displaystyle \int \phi _t \omega ^n =0$, $\omega _{\phi _t} > 0$, and $n \omega _{\phi _t} ^{n-1} - t\sum (n-k) \alpha _k \wedge \omega _{\phi _t} ^{n-k-1} > 0$. $\mathcal{T}$ is non-empty because at $t=0$ the equation is the usual Monge-Amp\`ere equation which has a solution thanks to \cite{Yau}. As usual, we need to prove that $\mathcal{T}$ is both, open and closed. \\

\textbf{Openness} : Let $\mathcal{C}$ be the set of $C^{2,\beta}$ zero-average functions $\phi$ such that $\omega _{\phi}>0$ where the background metric used to define the Banach spaces and the average is $\omega$. We proceed to define a smooth map $T$ from $\mathcal{B}$ (where $\mathcal{B}$ is an open subset of $\mathcal{C} \times [0,1]$ such that $n \omega _{\phi} ^{n-1} - t\sum (n-k) \alpha _k \wedge \omega _{\phi} ^{n-k-1} > 0$) to $C^{0,\beta}$ top forms $\gamma >0$ such that $\int \gamma = 1$. It is given by $T(\phi, t) = \displaystyle \frac{\omega _{\phi} ^n - t\displaystyle \sum _{k=1} ^{n-1} \alpha _k \omega _{\phi} ^{n-k}}{\int \omega ^n - t\displaystyle\sum _{k=1} ^{n-1} \int \alpha _k \omega ^{n-k}}$. The derivative $DT$ at the point $(\phi_a, a)$ evaluated on the vector $(u,0)$ is computed to be $DT_{\phi_a, a}(u,0) = \frac{(n \omega _{\phi _a} ^{n-1} - a\sum (n-k) \alpha _k \wedge \omega _{\phi _a} ^{n-k-1}) \wedge \ddc u}{\int \omega ^n - a\sum \int \alpha _k \omega ^{n-k}}$. It is  easily seen to be a self-adjoint elliptic operator. By the Fredholm alternative, we can solve the PDE if the right hand side is orthogonal to its kernel. Its kernel (by the maximum principle) consists of constants. Thus by the implicit function theorem on Banach manifolds, on the level set $T^{-1}\left(\frac{\alpha _n}{\int \alpha _n}\right)$ we can locally solve for $\phi$ as a smooth function of $t$.  \\

\textbf{Closedness} : If $t_j \rightarrow t$, we need to prove that a subsequence $\phi _{j} \rightarrow \phi$ in $C^{2,\gamma}$, $\omega _{\phi} > 0$, and $n \omega _{\phi} ^{n-1} - t\sum (n-k) \alpha _k \wedge \omega _{\phi} ^{n-k-1} >0$. By the usual bootstrap argument this implies that $\phi$ is smooth. The Arzela-Ascoli theorem shows that it is enough to prove \emph{a priori} $C^{2, \gamma}$ estimates in order to show convergence of $\phi _j \rightarrow \phi$. The following argument shows that the limiting $\phi$ satisfies the other conditions.
\begin{lemma}
If $\Vert \phi _j \Vert _{C^2} \leq C$, then $\omega _{\phi _j} \geq  R \omega > 0$ and $n \omega _{\phi _j} ^{n-1} - t_j\sum (n-k) \alpha _k \wedge \omega _{\phi _j} ^{n-k-1} \geq R\omega ^{n-1} >0$ where the positive constant $R$ depends on $C$.
\label{uniformellip}
\end{lemma}
\begin{proof}
Recall that by assumption $\alpha _{k_0} \geq \delta \omega ^{k_0}$ for some $k_0$ and some constant $\delta >0$. Equation \ref{conti} implies that
$$1\geq \delta \frac{\omega ^{k_0} \wedge \omega _{\phi _j} ^{n-k_0}}{\omega _{\phi _j} ^n} \geq \delta \frac{k_0 ! (n-k_0)!}{n!} \sum \frac{1}{\lambda _{i_1}\ldots\lambda_{i_{k_0}}}.$$
where $\lambda_{i}$ are the eigenvalues of $\omega_{\phi_j}$ with respect to $\omega$. Hence we see that $\omega _{\phi _j} \geq R \omega$ for some $R >0$ depending on the upper bound on $\omega _{\phi _j}$. Indeed, if the smallest eigenvalue $\lambda_{1}$ (with respect to $\omega$) of $\omega_{\phi_j}$ becomes arbitrarily small then the right hand side of the equation above becomes arbitrarily large because $\omega_{\phi_j}$ is bounded above by assumption. \\
\indent Likewise, at an arbitrary point $p$, let $v$ be a unit $(1,0)$-form with respect to $\omega$. Choose coordinates so that $\omega _{\phi_j}$ is diagonal with eigenvalues $1$ and $v=c\frac{\partial}{\partial z^1}$ at $p$. Note that $c$ is bounded below and above because $\omega _{\phi_j}$ is. Then
\begin{gather}
\displaystyle \left( n\omega _{\phi _j} ^{n-1} -  t_j\sum _k (n-k) \alpha _k \wedge \omega _{\phi _j} ^{n-k-1}\right ) \wedge v \wedge \bar{v} = c ^2 \left( n\omega _{\phi _j} ^{n-1} -  t_j\sum _k (n-k) \alpha _k \wedge \omega _{\phi _j} ^{n-k-1}\right ) _{2\bar{2}3 \bar{3} \ldots n \bar{n}} \nonumber \\
= c^2 n!  - c^2t_j \sum _k \sum _{\vert I \vert = n-k-1 , 1 \notin (I,I^{0})} (n-k)! \left(\alpha _k \right ) _{I^{0}} ,\nonumber
\end{gather}
where if the multi-index $I=(i_1,\ldots, i_{n-k-1})$ then $I^{0}$ is the multi-index consisting of $k$ other numbers in $1,2,\ldots,n$. Equation \ref{conti} implies that
\begin{gather}
n!  = t_j \sum _k \sum _{\vert V \vert = n-k} (n-k)! \left(\alpha _k \right ) _{V^{0}}  \nonumber \\
\Rightarrow  n!  - t_j \sum _k \sum _{\vert I \vert = n-k-1 , 1 \notin (I,I^{0})} (n-k)! \left(\alpha _k \right ) _{I^{0}}  = t_j \sum _k \sum _{\vert W \vert = n-k, 1 \notin W} (n-k)! \left(\alpha _k \right ) _{W^{0}} \nonumber \\
\geq t_j \delta (n-k_0)! \sum _{\vert W \vert = n-k_0, 1 \notin W} \left(\omega ^{k_0} \right ) _{W^{0}}\geq \tilde{R}. \nonumber
\end{gather}
The last equation implies that $\displaystyle \left( n\omega _{\phi _j} ^{n-1} -  t_j\sum _k (n-k) \alpha _k \wedge \omega _{\phi _j} ^{n-k-1}\right ) \wedge v \wedge \bar{v} \geq R$ for some $R>0$ depending on the upper bound on $\omega _{\phi _j}$.
\end{proof}

\begin{rema}\label{Weithmpf}
At this juncture, if in equation \ref{maineq} we substitute $\alpha _k = \psi _{\epsilon} c_{k} \chi ^k \ \forall \ 1 \leq k \leq n-1$, $\alpha _n = (c_n + \epsilon)\chi ^n $ (where $\chi$ is a K\"ahler metric), such that $c_k \geq 0 \ \forall \ 1\leq k \leq n$, $\epsilon >0$ and $\psi _{\epsilon} = \frac{\displaystyle \int \omega ^n - \int (\epsilon + c_n) \chi^n}{\displaystyle \int \omega ^n - \int c_n \chi^n}$ are constants such that $\displaystyle \sum_{k=1} ^{n} c_k > 0$ and $\displaystyle \int \omega ^n = \sum _{k=1} ^{n} \int c_k \chi ^k \wedge \omega ^{n-k}$, then the $C^{2,\gamma}$ \emph{a priori} estimates in \cite{Weisun,Coll} guarantee that the resulting equation has a smooth solution $\phi _{\epsilon}$. We note that the \emph{a priori}  estimates up to the second order in \cite{Weisun,Coll} do not depend on $\epsilon$. Actually, using the Evans-Krylov theorem and the fact that (by lemma \ref{uniformellip}) the equation is uniformly elliptic we have $C^{2,\gamma}$ \emph{a priori} estimates independent of $\epsilon$. Therefore, upto a subsequence $\phi _{\epsilon} \rightarrow \phi$ in $C^{2,\beta}$ as $\epsilon \rightarrow 0$. Hence we recover the main theorem in \cite{Weisun} in the K\"ahler case via a continuity path that passes only through K\"ahler metrics and more importantly, openness is easy to prove (as opposed to \cite{Weisun}).
\end{rema}
As mentioned earlier, lemma \ref{uniformellip} shows that all we have to do in order to solve equation \ref{maineq} is to prove \emph{a priori} $C^{2, \gamma}$ estimates on $\phi$. We prove a general $C^0$ estimate on $\phi$ here. \\

\textbf{The uniform estimate} : Before proceeding further, we prove a lemma about concavity of certain potentially non-symmetric functions of the K\"ahler form.
\begin{lemma}\label{concavitylem}
The function $\omega  \rightarrow \frac{\alpha _k \wedge \omega ^{n-k}}{\omega ^n}$ is a convex function of K\"ahler forms if $\alpha _k \geq 0$.
\end{lemma}
\begin{proof}
Recall that $\alpha _k = f \chi ^k + \displaystyle \sqrt{-1}^{n-1} (-1)^{n(n-1)/2} \sum _{i=1}^N f_i \Phi _i \wedge \bar{\Phi} _i$ where $f, f_i \geq 0$. Choosing normal coordinates for $\chi$ we see that $\omega \rightarrow f \frac{\chi^ k \wedge \omega ^{n-k}}{\omega ^n}$ is convex by standard theory. Furthermore, let $\omega _1$ and $\omega _2$ be two K\"ahler forms. At the point under consideration choose coordinates so that $\omega _1$ is Euclidean and $\omega _2$ is diagonal with eigenvalues $\lambda _j$.  Therefore for some positive constant $C$ we have,
\begin{gather}
\frac{\Phi _i \wedge \bar{\Phi}_i \wedge (t \omega _1 + (1-t) \omega _2)^{n-k}}{(t\omega _1 + (1-t) \omega _2)^n} = C\displaystyle \sum _{\vert I \vert = n-k} \left (\Phi _i \wedge \bar{\Phi}_i \right )_{I^{0}} \frac{(t+(1-t)\lambda)_{I}}{(t+(1-t)\lambda _1)\ldots (t+(1-t)\lambda _n)}. \label{subm}
\end{gather}
It is now easy to deduce the desired result from expression \ref{subm} and the fact that $\frac{1}{\det(A)}$ is convex as a function of positive-definite matrices $A$.
\end{proof}
\begin{prop}\label{uniformestimate}
A smooth solution of equation \ref{maineq} satisfies $\Vert \phi \Vert _{C^0 (X)} \leq C$ where $C$ depends only on  $\omega$, bounds on the coefficients of the equation, and the positive lower bound on $\alpha _{k_0}$.
\end{prop}
\begin{proof}
 We follow Yau's by-now-classical \cite{Yau} technique adapted from \cite{GenMA}. In whatever follows, unless otherwise specified, all controlled constants are denoted by $C$. Without loss of generality we may change the normalisation of $\phi$ so that $\sup \phi = -1$. Let $\phi = - \phi _{-}$. We will find an upper bound on $\phi _{-}$ using Moser iteration (as usual). Let $\Theta = \displaystyle \omega ^n -  \sum _{k=1} ^n \alpha _k \wedge \omega ^{n-k}$. Subtracting $\Theta$ on both sides from $\omega _{\phi}^n - \displaystyle \sum _{k=1} ^n \alpha _k \wedge \omega _{\phi} ^{n-k}=0$, multiplying by $\phi _{-} ^p$ and integrating we see that
\begin{gather}
\displaystyle \int _X \phi _{-} ^ p\int _{0}^{1} \frac{d}{dt} \left ( \omega _{t\phi}^n - \displaystyle \sum _{k=1} ^n \alpha _k \wedge \omega _{t\phi} ^{n-k} \right ) dt \leq C \Vert \phi _{-} \Vert _{L^p} ^p \nonumber \\
\Rightarrow -\displaystyle \int _{0}^{1} \int _X \phi _{-} ^ p \left ( n\omega _{t\phi}^{n-1} - \displaystyle \sum _{k=1} ^n (n-k) \alpha _k \wedge \omega _{t\phi} ^{n-k-1} \right ) \sqrt{-1}\pbp \phi_{-} dt \leq C \Vert \phi _{-} \Vert _{L^p} ^p \nonumber \\
\Rightarrow \displaystyle \int _{0}^{1} \int _X \sqrt{-1} \partial \phi _{-} ^ {\frac{p+1}{2}} \wedge \bar{\partial} \phi _{-} ^{\frac{p+1}{2}} \wedge \left ( n\omega _{t\phi}^{n-1} - \displaystyle \sum _{k=1} ^n (n-k) \alpha _k \wedge \omega _{t\phi} ^{n-k-1} \right ) dt \nonumber \\
\leq  \frac{C(p+1)^2}{p} \Vert \phi _{-} \Vert _{L^p} ^p \leq C(p+1) \Vert \phi _{-} \Vert _{L^p} ^p. \nonumber
\end{gather}
At a point $q$ we choose coordinates normal coordinates $z^i$ for $\omega$ so that $\partial \phi _{-}$ is proportional to $\frac{\partial} {\partial z^1}$. This means that at $q$
\begin{gather}
\partial \phi _{-} ^ {\frac{p+1}{2}} \wedge \bar{\partial} \phi _{-} ^{\frac{p+1}{2}} \wedge\left ( n\omega _{t\phi}^{n-1} - \displaystyle \sum _{k=1} ^n (n-k) \alpha _k \wedge \omega _{t\phi} ^{n-k-1} \right ) \nonumber \\
 = \partial \phi _{-} ^ {\frac{p+1}{2}}  \bar{\partial} \phi _{-} ^{\frac{p+1}{2}}  \left ( n\omega _{t\phi}^{n-1} - \displaystyle \sum _{k=1} ^n (n-k) \alpha _k \wedge \omega _{t\phi} ^{n-k-1} \right )_{2\bar{2}\ldots n\bar{n}} \nonumber \\
= \partial \phi _{-} ^ {\frac{p+1}{2}}  \bar{\partial} \phi _{-} ^{\frac{p+1}{2}}  \left(n\omega _{t\phi}^{n-1} \right)_{2\bar{2}\ldots n\bar{n}}\left ( 1 - \frac{(\displaystyle \sum _{k=1} ^n (n-k) \alpha _k \wedge \omega _{t\phi} ^{n-k-1})_{2\bar{2}\ldots n\bar{n}}}{(n\omega _{t\phi}^{n-1} )_{2\bar{2}\ldots n\bar{n}}} \right ). \nonumber
\end{gather}
Now we restrict ourselves to the subspace spanned by $\partial _2 , \ldots, \partial _n$. To emphasize this we denote the restriction of any form $\beta$ by $\tilde{\beta}$. Now we proceed as in \cite{Weisun}. Note that $\omega _{t\phi} = t \omega _{\phi} + (1-t) \omega$ and that $\det(A)^{1/n}$ is concave as a function of positive-definite $n\times n$ matrices $A$. Therefore $\tilde{\omega}_{t\phi} ^{n-1}\geq t^{n-1} \tilde{\omega}^{n-1} _{\phi}+ (1-t)^{n-1} \tilde{\omega} ^{n-1}$. Likewise, lemma \ref{concavitylem} shows that
\begin{gather}
1 - \frac{(\displaystyle \sum _{k=1} ^n (n-k) \alpha _k \wedge \omega _{t\phi} ^{n-k-1})_{2\bar{2}\ldots n\bar{n}}}{(n\omega _{t\phi}^{n-1} )_{2\bar{2}\ldots n\bar{n}}} \geq t \left ( 1 - \frac{\displaystyle \sum _{k=1} ^n (n-k) \tilde{\alpha} _k \wedge \tilde{\omega} _{\phi} ^{n-k-1}}{n\tilde{\omega} _{\phi}^{n-1}} \right ) \nonumber \\ + (1-t) \left ( 1 - \frac{\displaystyle \sum _{k=1} ^n (n-k) \tilde{\alpha} _k \wedge \tilde{\omega} ^{n-k-1}}{n\tilde{\omega} ^{n-1}} \right ) \nonumber
\end{gather}
Therefore,
\begin{gather}
\displaystyle  \frac{1}{n+1}\int _X \sqrt{-1} \partial \phi _{-} ^ {\frac{p+1}{2}} \wedge \bar{\partial} \phi _{-} ^{\frac{p+1}{2}} \wedge \left ( n\omega ^{n-1} - \displaystyle \sum _{k=1} ^n (n-k) \alpha _k \wedge \omega  ^{n-k-1} \right )  \leq  C(p+1) \Vert \phi _{-} \Vert _{L^p} ^p \nonumber \\
\Rightarrow  \Vert \nabla (\phi _{-} ^{(p+1)/2}) \Vert _{L^2} ^2 \leq C(p+1) \Vert \phi _{-} \Vert _{L^p} ^p. \nonumber
\end{gather}
From the this point onwards, the proof is standard. (See \cite{GenMA} for instance.)
\end{proof}

\textbf{Uniqueness} : If $\phi _1$ and $\phi _2$ are two smooth solutions of equation \ref{maineq} such that $\displaystyle \int \phi _1 \omega ^n = \int \phi _2 \omega ^n$, $\omega _{\phi _i} >0$ and $n\omega _{\phi} ^{n-1} - \displaystyle \sum _{k=1} ^{n-1} (n-k) \alpha _k \wedge \omega _{\phi _i}^{n-k} >0$, then upon subtraction we get
\begin{gather}
\displaystyle  \int _{0} ^{1} \frac{d}{dt} \left ( \omega _{t\phi_1 + (1-t)\phi _2} ^n - \sum _{k=1} ^n \alpha _k \wedge \omega _{t\phi_1 + (1-t)\phi _2} ^{n-k} \right ) dt = 0 \nonumber \\
\Rightarrow \displaystyle \left ( \int _{0} ^{1} \left ( n\omega _{t\phi_1 + (1-t)\phi _2} ^{n-1} - \sum _{k=1} ^{n-1} (n-k)\alpha _k \wedge \omega _{t\phi_1 + (1-t)\phi _2} ^{n-k-1} \right ) dt \right )\wedge \sqrt{-1} \pbp (\phi_1 - \phi _2) = 0 \label{uni}
\end{gather}
The proof of proposition \ref{uniformestimate} shows that equation \ref{uni} is elliptic. Thus the maximum principle implies that $\phi_1 = \phi _2$.
\section{The gradient estimate}\label{gradi}
From now onwards we restrict ourselves to solving a special case of equation \ref{maineq} on the K\"ahler manifold $(X,\omega)$ where $\omega$ satisfies the cone condition. Firstly, let $\frac{1}{b}\omega \leq \chi \leq b \omega$ be an arbitrary K\"ahler metric on $X$ having nonnegative bisectional curvature. As mentioned in the introduction, we aim at solving
\begin{gather}
\omega _{\phi} ^n = \alpha \wedge \omega _{\phi} + \eta, \label{maineqspecial}
\end{gather}
where $\alpha = f \chi ^{n-1} + \displaystyle \sum _{a=1} ^N f_a \Phi _a \wedge \bar{\Phi}_a $ and $\eta >0$. In addition we assume that $-C\alpha \leq \nabla _X \alpha \leq C\alpha$ where $X$ is a real $\chi$-unit vector. Also, from now onwards we write $\eta = h \chi^n$ where $h>0$ is a smooth function, $\alpha \wedge \beta$ locally as $\chi ^n A^{k\bar{l}} \beta _{k\bar{l}}$ for a non-negative matrix $A$, and $\omega$ locally as $\omega _{i\bar{j}} dz^i d\bar{z}^j$ where $\omega$ is used (by abuse of notation) to denote both, the K\"ahler potential as well as the metric itself. \\
\indent In order to prove a gradient bound on $\phi$ we use Blocki's technique \cite{Blocki}. Denote by $\nabla$ the Levi-Civita connection associated to $\chi$. Let $\psi = \ln (\vert \nabla  \phi \vert^2) - \gamma (\phi)$ where $\gamma (t) = \frac{1}{2} \ln (2x+1)$ is chosen so that $\gamma ^{'} >E>0$ and $-(\gamma ^{''}+(\gamma^{'})^2) > Q>0$ for two positive constants $E$ and $Q$. At the maximum point $p$ of $\psi$, $\nabla \psi = 0$ and $\psi _{k\bar{l}}$ is negative semi-definite. Without loss of generality we may assume that $\vert \nabla \phi \vert(p) \geq N$ for any $N$. Choosing normal coordinates for $\chi$ at $p$ so that $\omega _{\phi}$ is diagonal with eigenvalues $\lambda _i$ we obtain,
\begin{gather}
0 = \psi _k (p) = \frac{\displaystyle \sum _i \phi _{ik} \phi _{\bar{i}} + \phi_{i} \phi _{\bar{i}k}}{\vert \nabla \phi \vert ^2} - \gamma ^{'} \phi _k \nonumber \\
\psi _{k\bar{l}}(p) = -(\gamma ^{''}+(\gamma ^{'})^2) \phi _k \phi _{\bar{l}} - \gamma ^{'}\phi _{k\bar{l}} + \frac{\displaystyle \sum _i \phi _{ik\bar{l}} \phi _{\bar{i}} + \phi _{ik} \phi _{\bar{i} \bar{l}}  + \phi _{\bar{i}k} \phi _{i\bar{l}}+ \phi_{i} \phi _{\bar{i}k\bar{l}}}{\vert \nabla  \phi \vert ^2} + \frac{\chi ^{i \bar{j}} _{,k \bar{l}}\phi_i \phi_{\bar{j}}}{\vert \nabla  \phi \vert^2}\label{secondderiv}.
\end{gather}
Rewriting equation \ref{maineqspecial} as $1=\frac{\alpha \omega _{\phi}}{\omega _{\phi}^n} +\frac{\eta}{\omega _{\phi} ^n}$, differentiating once, and multiplying by $\phi_{\bar{i}}$, at the point $p$ we obtain the following (after using the assumption on $\nabla \alpha$).
\begin{gather}
0 \geq -L^{k\bar{l}}\phi_{\bar{i}} \left (\omega _{k\bar{l} i} + \phi _{k \bar{l} i} \right) -C\vert \nabla \phi \vert \label{eqnderi}
\end{gather}
where $L^{k\bar{l}} =  -\frac{A^{k\bar{l}}}{\lambda _1 \ldots \lambda _n}+\frac{\delta _{k \bar{l}}}{\lambda _k}$. Multiplying equation \ref{secondderiv} by $L^{k\bar{l}}$ and using equation \ref{eqnderi} we obtain,
\begin{gather}
0\geq QL^{k \bar{l}} \phi _{k} \phi _{\bar{l}}  -\gamma ^{'} L^{k\bar{l}}
(\omega_{\phi})_{k\bar{l}}+ \gamma ^{'} L^{k\bar{l}}\omega _{k \bar{l}} + L^{k \bar{l}}\frac{\chi ^{i \bar{j}} _{,k \bar{l}}\phi_i \phi_{\bar{j}}}{\vert \nabla  \phi \vert^2} -\frac{\displaystyle \sum _i L^{k\bar{l}} \left ( \omega _{k\bar{l}i} \phi _{\bar{i}} + \omega _{k\bar{l} \bar{i}} \phi_i\right)}{\vert \nabla \phi \vert^2} - \frac{C}{\vert \nabla  \phi \vert}.
\end{gather}
Noting that $L^{k\bar{l}}(\omega _{\phi})_{k\bar{l}} = -\frac{\alpha \omega _{\phi}}{\omega _{\phi}^n} + n$ (which is larger than $n-1$ and less than $n$), and the assumption on the bisectional curvature $\chi^{\al \bar{\al}}_{,\be \bar{\be}} \geq 0 \ \forall \ \al,\be$ we get
\begin{gather}
0 \geq QL^{k \bar{l}} \phi _{k} \phi _{\bar{l}} + \gamma ^{'} L^{k\bar{l}}\omega _{k \bar{l}} -\gamma ^{'}\left ( n - \frac{\alpha \omega _{\phi}}{\omega _{\phi} ^n}\right ) - 2\frac{\vert L^{k\bar{l}} \nabla \omega _{k\bar{l}}\vert}{\vert \nabla \phi \vert}  - \frac{C}{\vert \nabla  \phi \vert}. \nonumber
\end{gather}
Now we multiply on both sides by $\frac{\omega _{\phi} ^n}{n!}$ and define $\tilde{L}^{k \bar{l}} = \frac{\omega _{\phi} ^n}{n!} L^{k \bar{l}} =  A^{k\bar{l}}-\delta ^{k \bar{l}}\frac{\lambda _1 \ldots \lambda _n}{\lambda _k}$. We get
\begin{gather}
0 \geq Q\tilde{L}^{k \bar{l}} \phi _{k} \phi _{\bar{l}} + \left(\gamma ^{'} - \frac{C}{\vert \nabla  \phi \vert}\right)\tilde{L}^{k\bar{l}}\omega _{k \bar{l}} -\gamma ^{'}\left (n(h+A^{k\bar{l}}(\omega _{\phi})_{k\bar{l}}) - A^{k\bar{l}}(\omega _{\phi})_{k \bar{l}}\right )  - \frac{C \omega _{\phi}^n}{\vert \nabla  \phi \vert} \nonumber \\
\Rightarrow C \geq Q\tilde{L}^{k \bar{l}} \phi _{k} \phi _{\bar{l}} + \left(\gamma ^{'} - \frac{C}{\vert \nabla  \phi \vert}\right)\left ( -A^{k \bar{l}} \omega _{k \bar{l}} + \frac{n\omega _{\phi} ^{n-1}\omega}{\omega_{\phi}^n}(h + A^{k \bar{l}} (\omega _{\phi})_{k \bar{l}} )\right )\nonumber \\-(n-1)\gamma ^{'} A^{k\bar{l}}(\omega _{\phi})_{k \bar{l}}  - \frac{C  (h+A^{k\bar{l}}(\omega _{\phi})_{k\bar{l}})}{\vert \nabla  \phi \vert} . \label{gradint}
\end{gather}
Note that inequality \ref{gradint} implies that at $p$, the expression $\frac{n\omega _{\phi} ^{n-1} \omega}{\omega _{\phi} ^n}$ is bounded above. At $p$ if we can prove that $\tilde{L}^{k\bar{l}} \geq T\chi^{k \bar{l}} >0$ then we will have a gradient estimate on $\phi$. Actually, if we manage to prove that $\Delta _{\chi} \phi$ is bounded above, then by lemma \ref{uniformellip} we are done. If we just prove that $\frac{n\omega _{\phi} ^{n-1} \omega}{\omega _{\phi} ^n}> n-1+\ep$ for some uniform positive constant $\ep$ then inequality \ref{gradint} implies that $\omega _{\phi}^n$ is bounded above and hence by the lower bound on $\omega_{\phi}$ so is $\Delta _{\chi} \phi$ above. Indeed, the following lemma coupled with this observation completes the proof of the gradient estimate.
\begin{lemma}
If at a point $q$, $\omega _{\phi} \geq R \omega > 0$, $\Delta _{\chi} \phi \rightarrow \infty$ then $\frac{n\omega _{\phi} ^{n-1} \omega}{\omega _{\phi} ^n} > n-1 +\ep$ for some uniform positive constant $\ep$.
\label{implem}
\end{lemma}
\begin{proof}
Without loss of generality we assume that $\lambda _1 \geq \lambda _2\ldots \lambda _n$. Equation \ref{maineqspecial} implies that at least $\lambda _n$ is bounded above at $p$. Notice that the cone condition implies
\begin{gather}
(n\omega ^{n-1})_{2\bar{2}\ldots n\bar{n}} > n! A^{1\bar{1}}.
\end{gather}
By the classical Hadamard inequality for matrices (see \cite{Info} for instance) $\omega _{2\bar{2}} \ldots \omega_{n\bar{n}} \geq \frac{(\omega ^{n-1})_{2\bar{2}\ldots}}{(n-1)!}$. Thus
\begin{gather}
\omega _{2\bar{2}} \ldots \omega _{n\bar{n}} > A^{1 \bar{1}}. \label{phew}
\end{gather}
Solving for $\lambda_1$ from equation \ref{maineqspecial} we see that $\lambda _1 = \frac{f+\displaystyle \sum _{k \neq 1} A^{k \bar{k}}}{\lambda _2 \ldots \lambda _n - A^{1\bar{1}}}$. This coupled with the lower bound on the $\lambda_i$ and the assumption that $\lambda _1 \rightarrow \infty$ implies that $\lambda _2 \ldots \lambda _n \rightarrow A^{1\bar{1}}$. Therefore,
\begin{gather}
\frac{n\omega _{\phi}^{n-1} \omega}{\omega_{\phi}^n} = \displaystyle \sum _{k=1} ^n \frac{\omega _{k\bar{k}}}{\lambda _k} \rightarrow \displaystyle \sum _{k=2} ^n \frac{\omega _{k\bar{k}}}{\lambda _k} \nonumber \\
\geq (n-1)\displaystyle \left ( \Pi _{k=2} ^{n} \frac{\omega _{k \bar{k}}}{\lambda _k} \right )^{1/(n-1)} \rightarrow (n-1)\displaystyle \left (  \frac{\displaystyle \Pi _{k=2} ^{n}\omega _{k \bar{k}}}{A^{1\bar{1}}} \right )^{1/(n-1)}, \nonumber
\end{gather}
where we used the AM-GM inequality. Using inequality \ref{phew} we are done.
\end{proof}
\section{Higher order estimates}\label{high}
\indent In this section we prove the partial Laplacian estimate. In addition to the assumptions in section \ref{gradi} we assume that $\alpha$ is parallel with respect to $\chi$.
\begin{rema}\label{exam}
It is but natural to wonder whether there are any forms $\alpha$ that satisfy the desired requirements other than multiples of $\chi ^{n-1}$. If $X$ is a complex torus and $\chi$ the flat metric, then $\alpha = dz^{1}\wedge d\bar{z}^1 \ldots dz^{n-1} \wedge d\bar{z} ^{n-1}$ furnishes a non-trivial degenerate example. In general, one can take a locally hermitian symmetric space or a product of any $2$ manifolds with $\chi$ being the product metric to produce lots of examples using $(1,1)$-forms. (Note that in our case we also need $\chi$ to have nonnegative bisectional curvature.) In fact, it is known that on manifolds other than local products or locally hermitian symmetric spaces the only such forms are indeed multiples of $\chi^{n-1}$. According to Bryant \cite{Bry}, this result follows from the classification of Riemannian holonomy groups.
\end{rema}
\textbf{Partial Laplacian bound} : We now prove an upper bound on $\frac{\alpha\omega _{\phi}}{\chi^n}$. As in \cite{MAtorus} we use the function $\Psi = \frac{\alpha \omega _{\phi}}{\chi^n} - \mu \phi$ where $\mu$ is a constant that will be chosen later. If we prove that $\Psi$ is bounded above then we are done. As before, at the maximum point $p$ of $\Psi$, $\Psi _k =0$ and $\Psi _{k\bar{l}}$ is negative semi-definite. We choose normal coordinates for $\chi$ at $p$ and make sure that $\omega _{\phi}$ is diagonal at $p$ with eigenvalues $\lambda_i$. Differentiation of $\Psi$ yields the following.
\begin{gather}
0 = \nabla _k \Psi = \Psi _k = \frac{\alpha \omega _{\phi,k}}{\chi^n} - \mu \phi _k \nonumber  \\
\Psi _{k\bar{l}} = \nabla _{\bar{l}}\nabla _k \Psi =  \frac{\alpha  \nabla _{\bar{l}}\nabla _ k\omega _{\phi}}{\chi ^n} - \mu \nabla _{\bar{l}}\nabla _k \phi \nonumber \\
\geq \frac{\alpha  \omega _{\phi,k\bar{l}}}{\chi ^n} - \mu (\omega_{\phi})_{k\bar{l}} +\mu \omega _{k\bar{l}}  \label{psisecondone}
\end{gather}
Differentiating $1 = \frac{\alpha \omega _{\phi}}{\omega _{\phi} ^n} + \frac{\eta}{\omega _{\phi}^n}$ twice, multiplying by $A^{k\bar{l}}$ and summing over $k=l$ we obtain
\begin{gather}
0 = \frac{\alpha \omega _{\phi,k}}{\omega _{\phi}^n} -\frac{n\omega _{\phi}^{n-1}\omega _{\phi,k}}{\omega _{\phi}^n}+\frac{\eta_{,k}}{\omega_{\phi}^n} \label{onederiv} \\
0 \geq   \frac{A^{k\bar{l}}\alpha \nabla _{\bar{l}}\nabla _ k\omega _{\phi}}{\omega_{\phi}^n} - \frac{A^{k\bar{l}}n\omega_{\phi} ^{n-1} \nabla _{\bar{l}}\nabla _ k\omega _{\phi}}{\omega _{\phi}^n} - A^{k\bar{l}}\frac{\eta_{,k}}{\omega_{\phi}^n} \frac{n \omega _{\phi}^{n-1}\omega _{\phi,\bar{l}}}{\omega_{\phi}^n}  - A^{k\bar{l}}\frac{\eta_{,\bar{l}}}{\omega_{\phi}^n} \frac{n \omega _{\phi}^{n-1}\omega _{\phi,k}}{\omega_{\phi}^n}+\frac{\nabla _{\bar{l}}\nabla _ k \eta}{\omega_{\phi}^n} \nonumber \\
\geq \frac{A^{k\bar{l}}\alpha \omega _{\phi,k\bar{l}}}{\omega_{\phi}^n} - \frac{A^{k\bar{l}}n\omega_{\phi} ^{n-1} \omega _{\phi,k\bar{l}}}{\omega _{\phi}^n} 
- C\left (\frac{\alpha \omega _{\phi}}{\omega_{\phi}^n} + n\right ) - A^{k\bar{l}}\frac{\eta_{,k}}{\omega_{\phi}^n} \frac{n \omega _{\phi}^{n-1}\omega _{\phi,\bar{l}}}{\omega_{\phi}^n}  - A^{k\bar{l}}\frac{\eta_{,\bar{l}}}{\omega_{\phi}^n} \frac{n \omega _{\phi}^{n-1}\omega _{\phi,k}}{\omega_{\phi}^n}-\frac{C\eta}{\omega_{\phi}^n},
\label{eqsecondtemp}
\end{gather}
where we used lemma \ref{concavitylem}. At this juncture we use equations \ref{psisecondone} and \ref{onederiv} to get
\begin{gather}
C \geq  \frac{A^{k\bar{l}}\alpha \omega _{\phi,k\bar{l}}}{\omega_{\phi}^n} - \frac{A^{k\bar{l}}n\omega_{\phi} ^{n-1} \omega _{\phi,k\bar{l}}}{\omega _{\phi}^n}- \mu C \left (  \frac{\chi^n}{\omega _{\phi}^n} \right )^2 .
\label{eqsecond}
\end{gather}
We multiply equation \ref{psisecondone} by $L^{k\bar{l}}$ and sum to obtain (after substituting in \ref{eqsecond})
\begin{gather}
C \geq - \mu C \left (  \frac{\chi^n}{\omega _{\phi}^n} \right )^2 - L^{k\bar{l}} \mu(\omega _{\phi})_{k\bar{l}} + L^{k\bar{l}} \mu \omega _{k\bar{l}} \nonumber \\
\Rightarrow \frac{C}{\mu} \geq -C \left (  \frac{\chi^n}{\omega _{\phi}^n} \right )^2 - \left (n-\frac{\alpha \omega _{\phi}}{\omega _{\phi}^n}\right ) + \left ( \frac{n \omega _{\phi}^{n-1}\omega}{\omega _{\phi}^n}  - \frac{\alpha \omega}{\omega _{\phi}^n}\right )\nonumber \\
= -C \left (  \frac{\chi^n}{\omega _{\phi}^n} \right )^2 - \left (n\frac{\eta}{\omega_{\phi}^n} + (n-1)\frac{\alpha \omega _{\phi}}{\omega _{\phi}^n}\right ) + \left ( \frac{n \omega _{\phi}^{n-1}\omega}{\omega _{\phi}^n}  - \frac{\alpha \omega}{\omega _{\phi}^n}\right )\label{aftersub}
\end{gather}
Since $\eta > 0$ we know that $\omega _{\phi}^n$ is bounded from below. Moreover, $0\leq \frac{\alpha \omega _{\phi}}{\omega _{\phi} ^n} \leq 1$. Therefore $\frac{n\omega _{\phi}^{n-1} \omega}{\omega _{\phi}^n} < C$. This implies a lower bound, $\omega _{\phi} > R\omega >0$. Since we are assuming that $\Delta _{\chi} \phi \rightarrow \infty$,
\begin{gather}
\frac{C}{\mu} \geq  \frac{n\omega_{\phi}^{n-1}\omega}{\omega _{\phi}^n} +(1-n)\frac{\alpha \omega _{\phi}}{\omega _{\phi}^n} \nonumber
\end{gather}
Using lemma \ref{implem} we see that $\frac{n\omega_{\phi}^{n-1}\omega}{\omega _{\phi}^n} > n-1 + \epsilon$ for some uniform positive constant $\epsilon$. Therefore $\frac{C}{\mu} \geq \epsilon$. Choosing $\mu$ to be large enough we arrive at a contradiction. This proves the partial Laplacian estimate. \\
\indent In the case when $\alpha >0$ the partial Laplacian estimate implies an estimate on $\Delta _{\chi} \phi$.  \\

\textbf{$C^{2,\ga}$ estimates} : The previously established partial Laplacian bound when $\alpha >0$ implies by lemma \ref{uniformellip} that the equation is uniformly elliptic. In fact, it also implies that $1=\frac{\alpha \omega _{\phi}}{\omega _{\phi}^n}+\frac{\eta}{\omega_{\phi}^n}$ is uniformly elliptic. Lemma \ref{concavitylem} implies that the equation is also convex. Thus the (complex version of) the Evans-Krylov theory \cite{Siu} is applicable and furnishes a $C^{2,\ga}$ estimate. This completes the proof of theorem \ref{mainthm}.
\section{Applications}\label{app}
\subsection{Representation of the top Chern character} \label{rep}\textbf{} \\
\indent Given a $(k,k)$ form $\eta$ representing the $k$th Chern character class $[tr((\Theta)^k)]$ of a vector bundle on a compact complex manifold (where $\frac{\sqrt{-1}F}{2\pi} = \Theta$ and $F$ is the curvature of a connection), it is natural to ask whether there is a metric $h$ on the vector bundle whose induced Chern connection realises
\begin{gather}
\tr\left(\Theta ^k \right)=\eta. \label{Chern}
\end{gather}
As phrased this question seems almost intractable. It is not even obvious as to whether there is \emph{any} connection satisfying this requirement, leave aside a Chern connection. Work along these lines was done by Datta in \cite{Datta} using the h-principle. Therefore, it is more reasonable to ask whether equality can be realised for the top Chern character form. To restrict ourselves further we ask whether any given metric $h_0$ may be conformally deformed to $h=h_0 e^{-\phi}$ satisfying the desired requirement. In the case of a line bundle $L$ (where the only choice we have is conformal deformations) equation \ref{Chern} boils down to the PDE
\begin{gather}
\mathrm{tr}\left ( \Theta _0 + \frac{\sqrt{-1}}{2\pi} \pbp \phi \right ) ^n = \eta \nonumber
\end{gather}
When $\Theta _0 >0$ and $\eta>0$ this is the usual Monge-Amp\'ere equation solved by Yau \cite{Yau}. In general, one gets a complicated fully nonlinear PDE which reduces to equation \ref{maineq} in some cases. It is clear from the case of a line bundle that for the general case of a vector bundle, unfortunately quite a few potentially unnatural positivity requirements will have to be made on the curvature $\Theta_0$ and the form $\eta$. Note that the local problem was dealt with in \cite{Kryl}. In \cite{MAtorus} an existence result was proven on complex $3$-tori. Using the result in \cite{Weisun} we prove theorem \ref{Coro1}.\\

\emph{Proof of theorem \ref{Coro1}}: According to theorem $1.1$ in \cite{Weisun} the equation
\begin{gather}
\omega _{\phi}^n = e^b\displaystyle \psi \sum _{1}^n c_i \chi ^i \omega _{\phi}^{n-i} \label{weieqn}
\end{gather}
on a compact K\"ahler manifold $(X,\omega)$ has a unique smooth solution $\phi, b$ satisfying $\omega _{\phi} >0$ and $n \omega_{\phi}^{n-1} - \psi \displaystyle \sum _{i=1}^{n-1}(n-i) c_i \chi ^i \omega _{\phi}^{n-i-1}>0$ if there exists a smooth function $v$ such that
\begin{gather}
\omega _v ^n \leq \displaystyle \psi \sum _{i=i}^n c_i \chi^ i \omega _v ^{n-i}. \label{condit}
\end{gather}
Actually, if $\displaystyle \int _{X}\omega ^n = \displaystyle \int _X \psi \sum _{1}^n c_i \chi ^i \omega ^{n-i}$ and $\psi \chi^ i$ is closed, then one can choose $v=0$ and $\tilde{\psi}=e^{-\tilde{b}} \psi$ where $b$ is small enough for condition \ref{condit} to hold. This shows that under such circumstances, equation \ref{weieqn} has a unique solution with $b=0$. Indeed, expanding equations \ref{Cherneq} and \ref{weieqn} (with $b=0$)  we get
\begin{gather}
\displaystyle \sum _{r=0}^n \dbinom{n}{r} \frac{1}{k}\mathrm{tr}\left ( \Theta _0 ^r \right ) \left(\sqrt{-1} \pbp \phi \right)^{n-r}  - \frac{\eta}{k} =0 \nonumber \\
\displaystyle \sum _{r=0}^n \omega ^r \left(\sqrt{-1} \pbp \phi \right)^{n-r} - \psi \sum _{k=1} ^n\sum _{r=0} ^{n-k} c_k \chi^ k \dbinom{n-k}{r}\omega ^r \left(\sqrt{-1} \pbp \phi \right)^{n-k-r} = 0 \nonumber \\
\end{gather}
Comparing the two equations we get the desired result. \qed \\
\textbf{} \\
\subsection{A special Lagrangian type equation} \label{spec} \textbf{} \\
\indent According to superstring theory the spacetime of the universe is constrained to be a product of a compact Calabi-Yau three-fold and a four dimensional Lorentzian manifold. A ``duality" relates the geometry of this Calabi-Yau manifold with another ``mirror" Calabi-Yau manifold. From a differential geometry standpoint this maybe thought of (roughly) as a relationship between the existence of ``nice" metrics on a line bundle on one Calabi-Yau manifold  and special Lagrangian submanifolds of the other Calabi-Yau manifold. Using the Fourier-Mukai transform, Leung-Yau-Zaslow showed \cite{leung} that this mirror symmetry implies that equation \ref{jacobeq} ought to be satisfied in some cases. In \cite{special}, Jacob and Yau showed that given an ample line bundle $L$ over a compact K\"ahler manifold with non-negative orthogonal bisectional curvature, $L^k$ admits a solution to equation \ref{jacobeq}. However, the assumption of non-negative orthogonal bisectional curvature is not desirable if one wants to apply such a result to general Calabi-Yau manifolds. Here we attempt to partially address this issue by restricting our attention to $3$-folds.\\

\emph{Proof of theorem \ref{Coro2}} : Equation \ref{jacobeq} can be written using $\Theta = \sqrt{-1}F$ as
\begin{gather}
-\Theta ^ 3 +3\omega ^2 \Theta = \tan (\hat{\theta}) \left ( \omega ^3 - 3 \Theta ^2 \omega \right ). \nonumber
\end{gather}
Grouping terms together we see that it is equivalent to
\begin{gather}
\Omega _{\phi} ^3 - 3\omega ^2 \Omega _{\phi} \sec^2 (\hat{\theta}) -2 \omega ^3 \tan (\hat{\theta})\sec^2 (\hat{\theta}) = 0,
\end{gather}
where $\Omega _{\phi} = \Omega + \sqrt{-1} \pbp \phi = \Theta _0 -\omega \tan (\hat{\theta}) + \sqrt{-1} \pbp \phi$. Comparing this equation to the theorem in \cite{Weisun} we see that if $\Omega >0$, $\tan (\hat{\theta}) >0$, and the cone condition $\Omega ^2 - \omega ^2 \sec ^2 (\hat{\theta}) > 0$ is satisfied, then the equation has a unique smooth solution upto a constant multiple. \qed \\

\indent The conditions imposed on $\Omega$ in theorem \ref{Coro2} are reminiscent of the ``stability" condition in \cite{special}. Here is a concrete example of a Calabi-Yau manifold where the theorem is applicable : \\
\indent  Let $X$ be $\mathbb{C}{\Lambda} \times K$ where $K$ is a projective K3 surface with Picard group generated by an ample line bundle $L$ (for example $K$ can be a non-singular degree $4$ surface in $\mathbb{P}^3$. Let $\omega$ be the product of the flat metric on the torus and the unique Calabi-Yau metric in the K\"ahler class $[L]+\ep [\gamma]$ where $[\gamma]$ is any cohomology class and $\epsilon$ is chosen to be small enough (as to how small can be determined easily) for the example to work. Endow $L$ with a metric $h_0$ with positive curvature $F_0$. Choose $k$ to be large enough so that for $(L^k, h_0 ^k)$, $\tan (\hat{\theta_k}) > 0$. Indeed,
$$\tan (\hat{\theta_k}) = \frac{\displaystyle \int \left (k^3\Theta_0 ^3 - 3k\Theta_0 (\Theta _0 + \ep \gamma) ^2 \right ) }{\displaystyle \int \left (3k^2 \Theta _0 ^2 (\Theta _0 + \ep \gamma) - (\Theta _0 + \ep \gamma )^3 \right )}.$$
So if $k\geq 2$ and $\ep$ small enough, then $\tan (\hat{\theta})>0$. Note that as $k \rightarrow \infty$, $\tan (\hat{\theta}_k)$ grows linearly in $k$.\\
Notice that if $\Theta _0 > 2 \omega \tan (\hat {\theta})$ then the cone condition is definitely satisfied for large enough $k$. Indeed,
\begin{gather}
\Omega _0 ^2 - \omega ^2 \sec ^2 (\hat{\theta}) = \Theta _0 ^2 + \omega ^2 \tan ^2 (\hat{\theta}) -2 \Theta _0 \omega \tan (\hat{\theta}) - \omega ^2 \sec ^2 (\hat{\theta}) \nonumber \\
= \Theta _0 ^2  -2 \Theta _0 \omega \tan (\hat{\theta}) - \omega ^2 = \Theta_0 ( \Theta _0 - 2\omega \tan (\hat{\theta})) - \omega ^2 \nonumber
\end{gather}
Hence, our requirement boils down to making sure that $k\Theta _0 - 2 \omega \tan (\hat{\theta}_k) >0$ for large $k$.
\begin{gather}
k\Theta _0 - 2 \omega \tan (\hat{\theta}_k) = k\Theta _0 - 2 (\Theta _0 + \ep \gamma) \frac{\displaystyle \int \left (k^3\Theta_0 ^3 - 3k\Theta_0 (\Theta _0 + \ep \gamma) ^2 \right ) }{\displaystyle \int \left (3k^2 \Theta _0 ^2 (\Theta _0 + \ep \gamma) - (\Theta _0 + \ep \gamma )^3 \right )}
\end{gather}
For large enough $k$ we just have to look at
\begin{gather}
k\Theta _0 - \frac{2}{3}k (\Theta _0 + \ep \gamma) \frac{\displaystyle \int \Theta_0 ^3 }{\displaystyle \int \Theta _0 ^2 (\Theta _0 + \ep \gamma)}
\end{gather}
which is obviously positive if $\ep$ is small enough.
\section*{Acknowledgements}
 The author thanks Wei Sun for answering questions about his paper and the anonymous referee for useful suggestions.

\end{document}